\documentclass[a4paper,10pt]{amsart}
\usepackage[shortlabels]{enumitem}
\usepackage{amsthm,amsmath,amssymb,graphics,graphicx,hyperref,epstopdf,mathrsfs}
\usepackage{mathtools}
\usepackage{breqn}      
\usepackage{tikz}  
\usetikzlibrary{shapes.geometric}

\usetikzlibrary{calc,intersections,through}
  \usepackage{capt-of}

\title{Families of Finite Sets in which no Set is covered by the Union of the Others}
\author{Guillermo Alesandroni}
\address{2000 Rosario, Santa Fe, Argentina}
\email{guillea@okstate.edu, alesangc@wfu.edu, alesandronig@yahoo.com}

\newtheorem{theorem}{Theorem}[section]

\theoremstyle{definition}
\newtheorem{definition}[theorem]{Definition}

\DeclareMathOperator{\lcm}{lcm}

\begin{document}
\maketitle
\begin{abstract}
Let $\mathscr{F}$ be a finite nonempty family of finite nonempty sets. We prove the following: 
\begin{enumerate}[(1)]
\item $\mathscr{F}$ satisfies the condition of the title if and only if for every pair of distinct subfamilies $\{A_1,\ldots,A_r\}$, $\{B_1,\ldots,B_s\}$ of $\mathscr{F}$, $\bigcup\limits_{i=1}^r A_i \neq \bigcup\limits_{i=1}^s B_i$.
\item If $\mathscr{F}$ satisfies the condition of the title, then the number of subsets of $\bigcup\limits_{A\in \mathscr{F}} A$ containing at least one set of $\mathscr{F}$ is odd.
\end{enumerate}
We give two applications of these results, one to number theory and one to commutative algebra.
\end{abstract}

\section{Motivation and Definition}

This work is not to be confused with [2] or [3], two articles whose titles are almost identical to the title of this paper; but it is to be compared to them. In [2] and [3], Erdős, Frankl, and Furedi study a family $\mathscr{F}$ of sets in which no set is contained in the union of $r$ others, for various values of $r$. In this article, we address the particular case where $r$ is one less than the cardinality of $\mathscr{F}$. In other words, we consider a family of sets, such that no set is contained in the union of the others.

In spite of these similarities, our work is different from that of Erdős, Frankl, and Furedi both in its approach and scope. While they are interested in bounds on the cardinality of $\mathscr{F}$, we investigate intrinsic properties of $\mathscr{F}$. Once our main results are proven, we give two applications: one to number theory and one to commutative algebra.

In the next definition, we give a special name to the families to be studied in the present article.

\begin{definition}
Let $\mathscr{F}$ be a finite nonempty family of finite nonempty sets. We will say that $\mathscr{F}$ is a \textbf{dominant family of sets}, if no set in $\mathscr{F}$ is contained in the union of the others. Equivalently, we will say that $\mathscr{F}$ is dominant, if every set of $\mathscr{F}$ contains some element not shared by any of the other sets of $\mathscr{F}$.
\end{definition}

The term dominant family of sets comes from its algebraic counterpart; namely, a dominant ideal [1]. This is how these concepts are related.

Each monomial ideal can be represented as a squarefree monomial ideal using a technique known as polarization [5]. Suppose that, in its polarized form, a monomial ideal $M$ is minimally generated by squarefree monomials $m_1,\ldots,m_q$. We say that $M$ is a dominant ideal if, for each $m_i$ there is a variable $x_{m_i}$ that appears in the factorization of $m_i$ but not in the factorizations of $m_1,\ldots,\widehat{m_i},\ldots,m_q$.

Denote by $set(m_i)$ the set of variables that appear in the factorization of $m_i$. Then the condition above can be restated as follows: $M$ is dominant if each set $set(m_i)$ contains a variable $x_{m_i}$ that is not shared by any of the sets $set(m_1),\ldots,\widehat{set(m_i)},\ldots,set(m_q)$.

In other words, suppose that, in its polarized form, $M$ can be represented as $M=(m_1,\ldots,m_q)$. Then $M$ is a dominant ideal if and only if $\{set(m_1),\ldots,set(m_q)\}$ is a dominant family of sets.

Note: Free resolutions of monomial ideals is a core topic of commutative algebra. Perhaps, its most emblematic construction is the Taylor resolution [6], which yields a free resolution for each monomial ideal. The class of dominant ideals gives a complete characterization of when the Taylor resolution is minimal [1].

\section{Main results}

\begin{theorem}\label{Theorem 1}
$\mathscr{F}$ is a dominant family of sets if and only if for every pair of distinct subfamilies $\{A_1,\ldots,A_r\}$, $\{B_1,\ldots,B_s\}$ of $\mathscr{F}$, $\bigcup\limits_{i=1}^r A_i \neq \bigcup\limits_{i=1}^s B_i$.
\end{theorem}

\begin{proof}
First, let us consider the case in which $\mathscr{F}$ is dominant. If $\{A_1,\ldots,A_r\}$ and $\{B_1,\ldots,B_s\}$ are different subfamilies of $\mathscr{F}$, there must be a set $C$ that belongs to one of the subfamilies but not to the other; say that $C\in \{A_1,\ldots,A_r\}$. Since $\mathscr{F}$ is dominant, there is an element $x\in C$, not shared by any other set of $\mathscr{F}$. Therefore, $x\in \bigcup\limits_{i=1}^r A_i$ and $x\notin \bigcup\limits_{i=1}^s B_i$, which proves the first implication. 

Now, let us assume that $\bigcup\limits_{i=1}^r A_i \neq \bigcup\limits_{i=1}^s B_i$, for every pair of distinct subfamilies $\{A_1,\ldots,A_r\}$, $\{B_1,\ldots,B_s\}$ of $\mathscr{F}$. Let $A\in \mathscr{F}$. By hypothesis, $\bigcup\limits_{B\in\mathscr{F}\setminus\{A\}}B \neq \bigcup\limits_{B\in\mathscr{F}}B$. That is, $\bigcup\limits_{B\in\mathscr{F}\setminus\{A\}}B \neq \left(\bigcup\limits_{B\in\mathscr{F}\setminus \{A\}}B\right) \cup A$. Thus, $A\nsubseteq \bigcup\limits_{B\in\mathscr{F}\setminus \{A\}} B$, which proves the second implication.
\end{proof}

\begin{theorem}\label{Theorem 2}
Let $\mathscr{A} = \{A_1,\ldots,A_n\}$ be a dominant family of sets, and let $A = \bigcup\limits_{i=1}^n A_i$. Then, the number of subsets of $A$ that contain at least one member of $\mathscr{A}$ is odd.
\end{theorem}

\begin{proof}
Since the power set of $A$ has even cardinality, it is enough to show that the number of subsets of $A$ not containing any of the elements of $\mathscr{A}$ is odd. The proof is by induction on $n$. If $n=1$, then $\mathscr{A} = \{A_1\}$, $A = A_1$, and the number of subsets of $A$ that do not contain any sets of $\mathscr{A}$ is $2^{\mid A\mid} - 1$, which is odd. Suppose that the theorem holds for $n-1$. Let $\mathscr{B}$ be the family of all subsets of $A$ not containing any elements of $\mathscr{A}$. Since $\mathscr{A}$ is dominant, there exists an element $x\in A_1\setminus\bigcup\limits_{i=2}^n A_i$. Define $\mathscr{B}_{x^-} = \{B\in\mathscr{B}:x\notin B\}$, and $\mathscr{B}_{x^+} = \{B\in\mathscr{B}:x\in B\}$. Note that $\mathscr{B}$ is the disjoint union of $\mathscr{B}_{x^-}$ and $\mathscr{B}_{x^+}$. Consider the function
\[\begin{array}{lllll}
f&:&\mathscr{B}_{x^+}&\rightarrow&\mathscr{B}_{x^-}\\
&&B&\rightarrow&B\setminus\{x\}.
\end{array} \] 
It is clear that $f$ is one-to-one. Therefore,
\begin{align*}
\mid\mathscr{B}\mid&=\mid\mathscr{B}_{x^+}\mid + \mid f(\mathscr{B}_{x^+})\mid + \mid\mathscr{B}_{x^-}\setminus f(\mathscr{B}_{x^+})\mid\\
&=2\mid\mathscr{B}_{x^+}\mid + \mid \mathscr{B}_{x^-}\setminus f(\mathscr{B}_{x^+})\mid.
\end{align*} 
Thus, in order to prove that $\mid\mathscr{B}\mid$ is odd, it is enough to prove that $\mid\mathscr{B}_{x^-}\setminus f(\mathscr{B}_{x^+})\mid$ is odd. This fact follows from two assertions that we will consider previously.

The first assertion is as follows:
\[\mathscr{B}_{x^-}\setminus f(\mathscr{B}_{x^+}) =\left \{B\in \mathscr{B}_{x^-}: A_1\subseteq B\cup \{x\}\right\}.\]
To show this, choose a set $D\in\mathscr{B}_{x^-}$, such that $A_1\subseteq D\cup\{x\}$. We want to prove that $D\in\mathscr{B}_{x^-}\setminus f(\mathscr{B}_{x^+})$. Suppose not. Then there is a set $C\in \mathscr{B}_{x^+}$, such that $f(C)=D$. By construction, $D=C\setminus \{x\}$. Hence, $D\cup\{x\}=C$, which implies that $A_1\subseteq C$, a contradiction. We conclude that $\left \{B\in \mathscr{B}_{x^-}: A_1\subseteq B\cup \{x\}\right\}\subseteq \mathscr{B}_{x^-}\setminus f(\mathscr{B}_{x^+})$.
Conversely, if $D\in \mathscr{B}_{x^-}\setminus f(\mathscr{B}_{x^+})$, we must have that $D\cup\{x\}\notin \mathscr{B}_{x^+}$. Therefore, for some $j$, $A_j\subseteq D\cup \{x\}$, but $A_j\nsubseteq D$. Since $x\in A_1\setminus \bigcup\limits_{i=2}^n A_n$, it follows that $j=1$, and 
$D\in\left \{B\in \mathscr{B}_{x^-}: A_1\subseteq B\cup\{x\}\right\}$, which proves our assertion.

The second assertion is as follows: if $\mathscr{A}' = \{A_2\setminus A_1,\ldots,A_n\setminus A_1\}$, $A' = \bigcup\limits_{i=2}^n (A_i\setminus A_1)$, and 
$\Lambda = \left\{B\setminus A_1 : B\in\mathscr{B}_{x^-}\setminus f(\mathscr{B}_{x^+})\right\}$, then $\Lambda$ is the class of all subsets of $A'$ not containing any of the elements of $\mathscr{A}'$. 
To prove this fact, consider a set $B\setminus A_1 \in\Lambda$. Since $B\in\mathscr{B}_{x^-}\setminus f(\mathscr{B}_{x^+})$, $A_1\subseteq B\cup\{x\}$. For each $i=2,\ldots,n$, let $y_i\in A_i\setminus B$. Since $x\notin A_i$, $y_i\in A_i\setminus (B\cup\{x\})\subseteq A_i\setminus A_1$. And since $y_i\notin B$, 
$A_i\setminus A_1 \nsubseteq B\setminus A_1$. Therefore, $B\setminus A_1$ is a subset of $A'$ not containing any elements of $\mathscr{A}'$, which proves the first inclusion. To prove the other inclusion, suppose that $C$ is a subset of $A'$ not contining any of $A_2\setminus A_1,\ldots,A_n\setminus A_1$.
This implies that $C\cup (A_1\setminus\{x\})$ is a subset of $A$ not containing any of $A_1,\ldots,A_n$. Hence, $C\cup(A_1\setminus\{x\}) \in \mathscr{B}$ and, given that $x\notin C\cup(A_1\setminus \{x\})$, we have that $C\cup(A_1\setminus \{x\}) \in \mathscr{B}_{x^-}$. Finally, $A_1\subseteq \left[C\cup (A_1\setminus\{x\})\right]\cup \{x\}$, which implies that $C\cup(A_1\setminus\{x\}) \in \mathscr{B}_{x^-}\setminus f(\mathscr{B}_{x^+})$. Thus, $C=\left(C\cup (A_1\setminus\{x\})\right)\setminus A_1\in\Lambda$. This proves the second inclusion and, hence, our assertion.

Now we can complete the proof. Since $\mathscr{A}'$ is a dominant family of sets, of cardinality $n-1$, it follows that $\mid \Lambda \mid$ is odd, by induction hypothesis. But $\mid\mathscr{B}_{x^-}\setminus f(\mathscr{B}_{x^+})\mid = \mid\Lambda\mid$, which proves the theorem.
\end{proof}

Below we sketch an alternative proof to Theorem \ref{Theorem 2} (which was ingeniously created by Darij Grinberg). The notation $A$ and $\mathscr{A}$ is as in Theorem \ref{Theorem 2}. 

Let $p$ be the number of all pairs $(X, J)$ with $X \subseteq A$ and
$J \subseteq \mathscr{A}$ satisfying $X \supseteq \bigcup_{B \in J} B$.

We shall compute $p$ in two ways:

One way is: $p = \sum_{X \subseteq A} q_X$, where $q_X$ is the
number of all $J \subseteq \mathscr{A}$ satisfying $X \supseteq \bigcup_{B
\in J} B$. It is easy to see that $q_X = 2^m$, where $m$ is the number
of all $B \in \mathscr{A}$ satisfying $X \supseteq B$. Hence, $q_X$ is even
if and only if $X \supseteq B$ for some $B \in \mathscr{A}$. Therefore, $p$
is congruent modulo $2$ to the number of all subsets $X \subseteq A$
that contain none of the sets in $\mathscr{A}$.

Another way is: $p = \sum_{J \subseteq \mathscr{A}} r_J$, where $r_J$ is
the number of all $X \subseteq A$ satisfying $X \supseteq \bigcup_{B
\in J} B$. It is easy to see that $r_J = 2^r$, where $r$ is the size
of the set $A \setminus \bigcup_{B \in J} B$. Hence, $q_X$ is even if
and only if $A \setminus \bigcup_{B \in J} B$ is nonempty; in other
words, $q_X$ is even if and only if $J \neq \mathscr{A}$ (here we use the
dominance requirement). Thus, $p \equiv 1 \mod 2$.

Comparing these two results, we conclude that the number of
subsets of $A$ that contain none of the sets in $\mathscr{A}$ is
$\equiv 1 \mod 2$. In other words, this number is odd. Since $2^{|A|}$ is even, it follows that the
number of subsets of $A$ containing at least one member of $\mathscr{A}$ is also odd.

\section{Application to number theory}

Suppose that $a_1,a_2,a_3,a_4$ are $4$ squarefree numbers (i.e., each of them is the product of distinct primes). Label the vertices of a solid tetrahedron $T$ with $a_1,a_2,a_3,a_4$, as shown in the figure below.

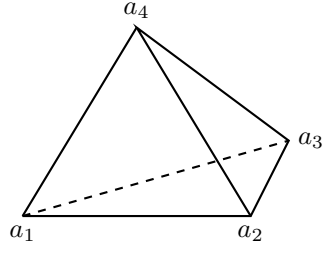
\[\begin{tikzpicture}

\draw[thick] (0,0) -- (3,0) -- (3.5,1)  -- (1.5,2.5) -- (3,0);
\draw[thick] (1.5,2.5)--(0,0);
\draw[thick, dashed] (0,0) -- (3.5,1) ;

\fill (0,0) circle (0pt) node[below] {$a_1$};
\fill (3,0) circle (0pt) node[below] {$a_2$};
\fill (1.5,2.5) circle (0pt) node[above] {$a_4$};
\fill (3.5,1) circle (0pt) node[right] {$a_3$};

\end{tikzpicture}\]

Next, label each edge of $T$ with the $\lcm$ of its $2$ vertices (for example, the edge with vertices $a_1,a_2$ gets the label $\lcm(a_1,a_2)$). Then, label each face of $T$ with the $\lcm$ of its $3$ vertices and, finally, label $T$ itself with the $\lcm$ of its $4$ vertices. That is, $T$ receives the label $a=\lcm(a_1,a_2,a_3,a_4)$. 
Each vertex, edge, and face of $T$, as well as $T$ itself, will be called a multiface. That is, the term multiface encompasses all objects that have been given a label.

We will consider the following problems.

(i) State a condition C such that, all multifaces have distinct labels if and only if condition C is satisfied. 

(ii) Suppose that condition C is satisfied. Let $N$ be the set of divisors of $a$ that are divisible by one of $a_1,a_2,a_3,a_4$. Is $\mid N\mid$ even or odd?

(i) For each squarefree integer $n$, denote by $set(n)$ the set of prime factors of $n$. Note that $set(\lcm(a_i,a_j)) = set(a_i)\cup set(a_j)$ and, more generally, $set(\lcm(a_{i_1},\ldots,a_{i_r})) = \bigcup\limits_{t=1}^r set(a_{i_t})$.
 Let $\lcm(a_{i_1},\ldots,a_{i_r})$, $\lcm(a_{j_1},\ldots,a_{j_s})$ be the labels of two multifaces. Then 
\begin{center}
$\lcm(a_{i_1},\ldots,a_{i_r})\neq \lcm(a_{j_1},\ldots,a_{j_s})$\\
\textit{if and only if}\\
 $set(\lcm(a_{i_1},\ldots,a_{i_r})) \neq set(\lcm(a_{j_1},\ldots,a_{j_s}))$\\
\textit{if and only if}\\
 $\bigcup\limits_{t=1}^r set(a_{i_t})\neq \bigcup\limits_{t=1}^s set (a_{j_t})$. 
\end{center}
Therefore, by Theorem \ref{Theorem 1}, 
\begin{center}
all multifaces have different labels\\
\textit{if and only if}\\
$ \{set(a_1),\ldots,set(a_4)\}\text{ is a dominant family of sets}$\\
\textit{if and only if}\\
each $set(a_i)$ contains a prime $p_i$ not shared by any of $set(a_1),\ldots,\widehat{set(a_i)},\ldots,set(a_4)$\\
\textit{if and only if}\\
there is a prime $p_i$ that divides $a_i$, but not any of the integers $a_1,\ldots,\widehat{a_i},\ldots,a_4$.
\end{center}

 In consequence, condition C can be stated as follows.

C: for each $i=1,\ldots,4$, there is a prime $p_i$ that is a factor of $a_i$ but not a factor of any of $a_1,\ldots,\widehat{a_i},\ldots,a_4$.

(ii) Consider the set $N = \{n\in\mathbb{N}: n\mid a \text{, and }a_i\mid n \text{, for some } 1\leq i\leq4\}$. Note that $n\in N$ if and only if $set(a_i)\subseteq set(n)\subseteq set(\lcm(a_1,\ldots,a_4)) = \bigcup\limits_{t=1}^4 set(a_t)$, for some $1\leq i\leq 4$. Hence, $\mid N\mid$ equals the number of subsets of $\bigcup\limits_{t=1}^4 set(a_t)$ that contain one set of $\mathscr{F} = \{set(a_1),\ldots,set(a_4)\}$. By Theorem \ref{Theorem 2}, $\mid N\mid$ is odd.

\section{Application to commutative algebra}

We close this paper with an application to free resolutions of monomial ideals. Even though we will use simple terms, some background material may be needed, and can be found in [4, 5].

Let $A=\{m_1,\ldots,m_q\}$ be a set of monomials, and denote $m=\lcm(A)$. We will say that a subfamily $\{m_{i_1},\ldots,m_{i_j}\}$ of $A$ is a \textbf{generator of $m$ in $A$} (or just a generator of $m$), if $\lcm(m_{i_1},\ldots, m_{i_j})=m$. In addition, a generator of $m$, none of whose proper subsets is a generator of $m$, will be called a \textbf{minimal generator of $m$ in $A$}.

Consider the class $\mathscr{F} = \{A_1,\ldots,A_r\}$ of all minimal generators of $m$ in $A$, and suppose that $A=\bigcup\limits_{i=1}^r A_i$. Notice that every generator of $m$ must be a subset of $A$ that contains one of $A_1,\ldots,A_r$. By Theorem \ref{Theorem 2}, if $\mathscr{F}$ is a dominant family of sets, the number of generators of $m$ is odd.

We will now explain how this result fits in the study of monomial resolutions. One important question in this topic is: if $m$ is a multidegree of the Taylor resolution of a monomial ideal $M$, does the minimal resolution of $M$ contain basis elements of multidegree $m$?

It is known that the minimal resolution of $M$ can be obtained from its Taylor resolution by doing a series of computations known as consecutive cancellations [5], each of which eliminates two basis elements of equal multidegree. This amounts to saying that the number of basis elements of multidegree $m$ in Taylor's resolution, and the number of basis elements of multidegree $m$ in the minimal resolution of $M$ differ by an even number. Therefore, if the number of basis elements of multidegree $m$ in the Taylor resolution is odd, then the minimal resolution of $M$ will contain at least one element of multidegree $m$.

And how can we tell if the number of basis elements of multidegree $m$ in the Taylor resolution is odd? Informally speaking, these basis elements are presicely the sets that we call generators of $m$ in $A$, where $A$ is the set of all minimal generators of $M$ that divide $m$ [4]. (Note that the minimal generators of $M$ are monomials, while the minimal generators of $m$ are sets of monomials.) It is here where Theorem \ref{Theorem 2} becomes a useful tool.

Summarizing, suppose that $m$ is a multidegree that occurs in the Taylor resolution of a monomial ideal $M$. Let $A$ be the class of all minimal generators of $M$ that divide $m$. Suppose that $A=\bigcup\limits_{i=1}^r A_i$, where $\mathscr{F} = \{A_1,\ldots,A_r\}$ is the class of all minimal generators of $m$. If $\mathscr{F}$ is a dominant family of sets, then the number of basis elements of multidegree $m$ in the Taylor resolution of $M$ is odd, and hence, the minimal resolution of $M$ contains a least one basis element of multidegree $m$.

For example, consider the ideal $M=( x_1x_2,x_3x_4,x_5x_6,x_1x_3x_5,x_2x_4x_6,x_7)$, and let $m= x_1x_2x_3x_4x_5x_6x_7$, which is a multidegree of the Taylor resolution of $M$. Then, $A=\{ x_1x_2,x_3x_4,x_5x_6,x_1x_3x_5,x_2x_4x_6,x_7\}$ (and $m=\lcm(A)$). It is not difficult to see that $A_1=\{ x_1x_2,x_3x_4,x_5x_6,x_7\}$ and $A_2=\{x_1x_3x_5,x_2x_4x_6,x_7\}$ are the only minimal generators of $m$ in $A$. Since $A=A_1 \cup A_2$, and given that $\mathscr{F} = \{A_1,A_2\}$ is a dominant family of sets, we conclude that the minimal resolution of $M$ contains at least one basis element of multidegree $m$. 

 \bigskip

\noindent \textbf{Acknowledgements}: A few days after I posted the first version of this paper to the arxiv, Darij Grinberg emailed me an elegant alternative proof to Theorem \ref{Theorem 2}, which is included in this new version. Thanks, Darij! My wife Danisa always types and proofreads my papers, which I deeply appreciate. Thanks, sweetheart!

\end{document}